\documentclass[12pt]{article}

\usepackage{amsmath,amsthm,amssymb,amsfonts}
\usepackage{latexsym,enumerate}
\usepackage{epsfig,color}

\newtheorem{theorem}{Theorem}[section]
\newtheorem{lemma}[theorem]{Lemma}
\newtheorem{proposition}[theorem]{Proposition}
\newtheorem{corollary}[theorem]{Corollary}

\theoremstyle{definition}

\newtheorem{remark}[theorem]{Remark}

\title{Removable singularities for 
degenerate elliptic Pucci operators}

 \author{Giulio Galise$^{\small 1}$ and Antonio Vitolo$^{\small 2}$}

\date{}

\begin{document}

\maketitle

 \hskip0.5cm \footnote{e-mail:{\tt  ggalise@unisa.it}}
Dipartimento di Matematica, {\it Sapienza Universit\`{a} di Roma},
\vskip0.05cm 
\hskip1.1cm Piazzale Aldo Moro, 5 - 00185 ROMA -  ITALY
\vskip0.1cm
 \hskip0.5cm \footnote{e-mail:{\tt  vitolo@unisa.it}}
Dipartimento di Matematica, {\it Universit\`{a} di Salerno},
\vskip0.05cm 
\hskip1.1cm via Giovanni Paolo II, 132 - 84084 Fisciano (SA) -  ITALY\\

\begin{abstract}
\noindent In this paper we introduce some fully nonlinear second order operators defined as weighted partial sums of the eigenvalues of the Hessian matrix, arising in geometrical contexts, with the aim to extend maximum principles and removable singularities results to cases of highly degenerate ellipticity.
\end{abstract}

\paragraph{AMS Classification:} 35J60,   35B50, 35B60, 35D40

\paragraph{Keywords:} degenerate elliptic equations, maximum principles, removable singularities, viscosity solutions

\section{Introduction}

In this paper we are concerned with Maximum Principle type results for viscosity solutions of fully nonlinear degenerate elliptic equations
\begin{equation}\label{eq1intr}
F(u,Du,D^2u)=f(x)
\end{equation}
in a bounded domain $\Omega\subset\mathbb R^n$. Here $Du$, $D^2u$ denote respectively the gradient and the Hessian matrix of the real valued function $u$ and  $f\in C(\Omega)$. The nonlinear structure assumptions on $F:\mathbb R\times\mathbb R^n\times{\cal S}^n\mapsto\mathbb R$ will be introduced below, see (\ref{SC}). \\
The main question addressed here, with the aim to deduce a corresponding removability result of singular sets, is the following: given a subset $E$ of the boundary $\partial\Omega$ 
, find conditions on the \lq\lq size\rq\rq\, of the \lq\lq singular\rq\rq \, set $E$ in order that the sign of a bounded subsolution of (\ref{eq1intr}) propagates from $\partial\Omega\backslash E$ inside $\Omega$, that is
\begin{equation}\label{eq2intr}
u\leq0\;\;\text{on}\;\;\partial\Omega\backslash E\;\;\text{and}\;\;\sup_\Omega u^+<\infty\;\;\Rightarrow\;\;u\leq0\;\;\text{in}\;\;\Omega.
\end{equation}
In a previous paper \cite{AGV} the authors dealt with the uniformly elliptic version of (\ref{eq1intr})-(\ref{eq2intr}), establishing sufficient conditions for the validity of (\ref{eq2intr}) based on the notion of capacity. In the spirit of Potential Theory they also proved  a corresponding removability singularities result. \\
These arguments apply to a class of pure second order degenerate equations whose prototypes are 
\begin{equation}\label{eq3intr}
{\cal P}^\pm_p(D^2u)=f(x),
\end{equation}
where ${\cal P}^\pm_p(X)$, $n\geq p\in\mathbb N$, are real mapping defined for $X \in {\cal S}^n$, respectively, as the maximal and minimal partial sums 
\begin{equation}\label{eq4intr}
\begin{split}{}
{\cal P}^+_p(X)&=\,e_{n-p+1}(X)+\ldots+e_{n}(X),\\ 
{\cal P}^-_p(X)&=\,e_{1}(X)+\ldots+e_{p}(X),\\
\end{split}
\end{equation}
of the ordered eigenvalues $e_{1}(X)\leq\ldots\leq e_{n}(X)$ of matrix $X$. \\
Equation (\ref{eq3intr}) has been deeply studied, in the general framework of their subequations theory, by Harvey and Lawson \cite{HL1},\cite{HL2},\cite{HL3}, to which we refer for a large numbers of properties of (\ref{eq4intr}) as well as to  the papers of Caffarelli, Li and Nirenberg \cite{CLNI} and \cite{CLN2}, which contain further maximum principles and removability results related to equation (\ref{eq3intr}).

\noindent Here we introduce a larger class of fully nonlinear possibly degenerate operators, for a positive integer $p \le n$, which we call the Pucci maximal and minimal operators of order $p$, respectively: 
\begin{equation}\label{Pucci-p-eq}
\begin{split}{}
{\cal P}^+_{\lambda,\Lambda|p}(X) &=\, \Lambda \sum_{i=n-p+1}^n e_i^+(X)-\lambda \sum_{i=n-p+1}^n e_i^-(X),\\
{\cal P}^-_{\lambda,\Lambda|p}(X) &=\, \lambda \sum_{i=1}^p e_i^+(X)-\Lambda \sum_{i=1}^p e_i^-(X)\,.
\end{split}
\end{equation}
We notice that, if $\lambda=\Lambda=1$, we obtain the previous operators ${\cal P}^\pm_{1,1|p}={\cal P}^\pm_p$ of Harvey and Lawson \cite{HL1}, a sort of truncated Laplacians,  and ultimately the Laplace operator when $p=n$. 
We refer to the next Section for a more detailed description of degenerate Pucci operators of order $p$.\\
Regarding the nonlinearity $F$ in (\ref{eq1intr}), the following structure condition is assumed throughout the paper:
\begin{equation}\label{SC}
F(t,\eta,Y)-F(s,\xi,X) \le {\cal P}^+_{\lambda,\Lambda|p}(Y-X)+b|\eta-\xi|-c(t-s)
\end{equation}
for all $(s,t,\xi,\eta,X,Y) \in (\mathbb R)^2\times (\mathbb R^n)^2\times ({\cal S}^n)^2$, where $b$ and $c$ are non negative constant and $0<\lambda\leq\Lambda$ .\\ As usual, eventually adding a non-zero term in $f(x)$, we will always impose $F(0,0,0)=0$. Making these assumptions, we think to the maximal equation
\begin{equation}\label{eq5intr}
{\cal P}^+_{\lambda,\Lambda|p}(D^2u)+b|Du|-cu=f(x)\,,\quad b\geq0,\,c\geq0,
\end{equation}
as our model.\\
We  also point out here that we can extend the notion of fundamental solution, which is well known for the standard Laplace and Pucci operators, see \cite{ASiSm},\cite{Lab1}. Setting
\begin{equation}\label{exponent}
\alpha^* = \frac{\lambda}{\Lambda}(p-1) -1\geq0, 
\end{equation}
the functions
\begin{equation}
\Phi_{\alpha^*}(x)= \left\{\begin{array}{ll}
|x|^{-\alpha^*} & \alpha^*  > 0\\
\log(\frac{|x|}{R})^{-1} & \alpha^*  = 0
\end{array}\right.
\end{equation}
are classical positive solutions of the degenerate Pucci maximal equation ${\cal P}^+_{\lambda,\Lambda|p}(D^2\Phi)=0$ in the punctured space $\mathbb R^n\backslash \{0\}$ and in the punctured ball $B_R=\{0<|x| < R\}$, respectively, which blow up at the origin.\\
 This is the basic tool for extended maximum principles and removability results, which we investigate in the present paper for the wide class of degenerate elliptic equations (\ref{eq1intr}) satisfying (\ref{SC}),
shedding light on the fact that in the degenerate case $p<n$ with $c=0$ the gradient term is competitive with respect to the second order term and a large coefficient $b$ may invalidate the Maximum Principle. \\

\noindent We collect our principal results in the following statements.\\

\noindent In the first one, we get an extended maximum principle outside a set $E$ of type $F_\sigma$, i.e. countable union of compact sets.
\begin{theorem}\label{th1}
Let $\lambda \le\Lambda$ positive real numbers and $p$ a positive integer such that $p \le n$, and suppose
\begin{equation}\label{alpha*}
\alpha^*=\frac\lambda\Lambda\,(p-1)-1\ge 0\,.
\end{equation}
 Let $\Omega$ be a bounded domain of $\mathbb R^n$ such that $\Omega \subseteq B_\delta=\{|x|<\delta\}$ and assume $\partial \Omega = E \cup E'$ with $E$ an
$F_\sigma$-set such that ${\rm{Cap}}_\alpha(E)=0$ for $\alpha \in [0,\alpha^*]$. Suppose that $f\in C(\Omega)$.  Assume that $u\in USC(\Omega)$ is bounded above and 
$${{\cal P}^+_{\lambda,\Lambda|p}}(D^2u)+b|Du|-cu\geq f(x)\quad in\,\,\Omega$$
in the viscosity sense, where $b$ and $c$ are non negative real numbers. \\
Let us consider the following cases: $(i)$ $c=0$; $(ii)$ $c>0$\,.\\
Case $(i)$. If $b\delta <\lambda p$, then
\begin{equation}\label{MPth1}
\sup_\Omega u\leq\limsup_{y\to E'}u(y)+C\left\|f^-\right\|_\infty\,,
\end{equation}
where $C$ is a constant depending only on $\lambda$, $p$, $b$ and $\delta$.\\
Case $(ii)$. For all $b \ge0$
\begin{equation}\label{MP+}
\sup_\Omega u\leq\limsup_{y\to E'}u^+(y)+C\left\|f^-\right\|_\infty\,,
\end{equation}
where now the constant $C$  depends also on $c$.\\
Estimates (\ref{MPth1}) and (\ref{MP+}) hold true in general for $0\le \alpha < \alpha^*$ and also for $\alpha=\alpha^*$ if $b=0$. 
\end{theorem}

\noindent Using the terminology of \cite{AGV}, Theorem \ref{th1} asserts that subsets $E\subset\partial\Omega$ of vanishing $\alpha$-capacity are \emph{exceptional}, intending that the behavior of bounded subsolutions near such sets does not influence the sign inside $\Omega$. This kind of result, called Extended Maximum Principle (EMP for short), is well known with $\alpha=n-2$ in the case of the Laplace operator $\Delta u= {\cal P}^+_{1,1|n}(D^2u)$, see for instance \cite[Theorems 3.4, 3.5]{L}, \cite[Theorem 5.16]{HK}. \\
An analogous result has been proved in \cite{AGV} in the case of fully nonlinear uniformly elliptic operators of type ${\cal P}^+_{\lambda,\Lambda|n}(D^2u)+b|Du|$  with $\alpha^* =\frac\lambda\Lambda(n-1)-1$, without restrictions on the product $b\delta$, as well as with $\alpha^*=p-2$ for the partial Laplacian ${\cal P}^+_{1,1|p}(D^2u)$.\\
EMP can be used as a tool to deal with the removable singularities problem of finding what conditions can be put on the size of $E\subset\Omega$ to ensure that any bounded solution $u$ of (\ref{eq1intr}) in $\Omega\backslash E$ can be extended to a solution $\tilde u$ of the same equation in the whole $\Omega$. If we consider the Laplace equation a very classical result establishes that every bounded harmonic function in $\Omega\backslash E$, $E$ compact, has a continuation which is harmonic in $\Omega$ provided the Riesz capacity ${\rm Cap}_{n-2}(E)=0$, see \cite[Theorem 3.3]{L}, \cite[Theorem 5.18]{HK}.  \\
Removable singularities for elliptic equations were investigated by many authors in different frameworks. We list some papers referring to the references therein for a complete account to the problem. J. Serrin \cite{Serrin1} obtained  results in this direction for weak solutions of linear uniformly elliptic equation of second order with H$\rm\ddot{o}$lder continuous coefficients, while in \cite{Serrin2} a wide class of quasilinear equations in divergence form is considered, so generalizing the removability result of De Giorgi and Stampacchia \cite{DGS} concerning solutions of the minimal surfaces equations. In a well known paper of Brezis and Nirenberg \cite{BrNi} no a priori assumptions are made about the behavior of the solution $u$ near the singular set. Veron's monograph \cite{Ve} is a full description of topics about singular solutions.\\
In the fully nonlinear viscosity setting Labutin \cite{Lab2} introduced a suitable notion of capacity in order to give a complete characterization of removable sets for pure second order uniformly elliptic equations, while in \cite{AGV} sufficient conditions for the removability of solutions of nonlinear equations depending also on lower order terms are obtained by using the classical Riesz capacity: singular sets $E$ having ${\rm Cap}_\alpha(E)=0$ are removable for $0\leq\alpha<\frac\lambda\Lambda(n-1)-1$ and for $\alpha=\frac\lambda\Lambda(n-1)-1$ in the pure second order case. Regarding the degenerate partial Laplacians ${\cal P}^\pm_p$, a geometric approach was developed in \cite{HL2}. \\

\noindent In the next theorem we generalize the above results to the wide class (\ref{eq1intr}) of degenerate elliptic equations, and the domain $\Omega$ has not to be necessarily bounded.

\begin{theorem}\label{threm}
Let $\Omega$ be a domain of $\mathbb R^n$, and $E$ be a closed subset in the relative topology of $\Omega$. 
Suppose that $F=F(t,\xi,X)$ is  a degenerate elliptic operator satisfying the structure condition (\ref{SC}).\\ 
If $u$ is a viscosity solution in $\Omega\backslash E$, bounded on bounded sets, of equation
\begin{equation}\label{eq1threm}
F(u,Du,D^2u)=f(x),
\end{equation}
 where $f\in C(\Omega)$, and the $\alpha$-Riesz capacity ${\rm{Cap}}_\alpha(E)=0$ for $\alpha \in[0,\alpha^*]$, 
then $u$ can be extended to a solution $\tilde{u}\in C(\Omega)$ of $(\ref{eq1threm})$ in the whole $\Omega$. 
 This holds true in general for $0\le \alpha < \alpha^*$ and also for $\alpha=\alpha^*$ if $b=0$.
\end{theorem}

\noindent The paper is organized as follows: in Sections \ref{Pucci} and \ref{Viscosity} we  report on definitions and some properties  concerning viscosity solutions and the operators ${\cal P}^\pm_{\lambda,\Lambda|p}$. Section \ref{A priori bound} is devoted to a priori estimates for subsolutions of (\ref{eq5intr}), and we also show a counterexample to the validity of Maximum Principle.  In Section \ref{Extended a priori bound} we establish the extended version of these results, while in Section \ref{Removable singularities} we deal with removable singularities.

\section{Degenerate Pucci operators}\label{Pucci}

Let $({\cal S}^n,\leq)$ be the partial ordered set of $n\times n$ real symmetric matrices in which $X\leq Y$  means that $Y-X$ is a positive semidefinite matrix. The eigenvalues of a matrix $X \in {\cal S}^n$ will be arranged in increasing order: $e_1(X) \le e_2(X) \le \ldots \le e_n(X)$. We will consider the norm 
\begin{equation} 
\left\|X\right\|=\sup\left\{|e_i(X)|:\,i=1,\dots,n\right\}
\end{equation}
Let $\Omega$ be an open connected set of $\mathbb R^n$. A continuous real valued mapping $F:\Omega \times \mathbb R\times\mathbb R^n\times{\cal S}^n\mapsto\mathbb R$ is said to be \emph{degenerate elliptic} if it is non-decreasing in the matrix variable: {for any $(x,s,\xi)\in\Omega\times\mathbb R\times\mathbb R^n$}
\begin{equation*}
{F(x,s,\xi,X)\leq F(x,s,\xi,Y)\quad\text{whenever $X\leq Y$.}}
\end{equation*}
In what follows, unless otherwise stated, $F$ will be assumed \emph{proper}, i.e. degenerate elliptic and non-increasing in the scalar variable for any $(x,\xi,X)\in\Omega\times\mathbb R^n\times {\cal S}^n$:
\begin{equation} 
F(x,s,\xi,X)\geq F(x,t,\xi,X)\quad\text{whenever $s\leq t$}.
\end{equation}
{Moreover, $F$ is \emph{uniformly elliptic} if there exist two constants $\Lambda\geq\lambda>0$, called ellipticity constants, for which
\begin{equation*}
\lambda Tr(P)\leq F(x,s,\xi,X+P)-F(x,s,\xi,X)\leq\Lambda Tr(P)\quad \forall\text{ $P\geq0$}
\end{equation*}
and all $(x,s,p,X)\in\Omega\times\mathbb R\times\mathbb R^n\times{\cal S}^n$.\\ }

\noindent Let $W$ be a linear subspace of $\mathbb R^n$, and let $P_{_W}$ be the orthogonal projection operator on $W$, represented with respect to the standard basis  
$$
E^1= \left(\begin{array}{l}
1\\
\vdots\\
0
\end{array}\right)\,,\, \dots\,,\, E^n= \left(\begin{array}{l}
0\\
\vdots\\
1
\end{array}\right)\,
$$
We define the projection of $A \in {\cal S}^n$ on $W$ setting
\begin{equation}
A_{_W}:= P_{_W}AP_{_W}\,,
\end{equation}
which represents the restriction of the quadratic form $A$ on $W$. Note that $P_{_W}=BB^T$, where the columns of $B$ are an orthonormal basis of $W$, and, if $A=I$, the $n\times n$ identity matrix, then $I_{_W}=P_{_W}$. We will use in the sequel the fact that the mapping $A \in {\cal S}^n \to A_W \in {\cal S}^n$ is linear, and also that $I_{_W}^2=I_{_W}$.\\
We define the {\it degenerate elliptic Pucci maximal and minimal operators restricted to $W$}, with ellipticity constants $\lambda \in (0,\infty)$ and $\Lambda\in [\lambda,\infty)$, respectively as 
\begin{equation}\label{Pucci-W-eq}
\begin{split}
{\cal P}^+_{\lambda,\Lambda|W}(X) = \Lambda\,Tr(X^+_{_W})-\lambda\,Tr(X^-_{_W})\,,\\
{\cal P}^-_{\lambda,\Lambda|W}(X) = \lambda\,Tr(X^-_{_W})-\Lambda\,Tr(X^+_{_W})\,,
\end{split}
\end{equation}
where $X^\pm$ are the unique non-negative matrices such that $X=X^+-X^-$ and $X^+X^-=0$, while $Tr(X)$ denotes the trace of the matrix $X$ and $X^\pm_{_W}$ stands for $\left(X_{_W}\right)^\pm$. Note that, if $W=\mathbb R^n$, the above definition returns the standard Pucci extremal operators \cite{P}:
\begin{equation*}
\begin{split}
{\cal M}^+_{\lambda,\Lambda}(X) = \Lambda\,Tr(X^+)-\lambda\,Tr(X^-)\,,\\
{\cal M}^-_{\lambda,\Lambda}(X) = \lambda\,Tr(X^-)-\Lambda\,Tr(X^+)\,.
\end{split}
\end{equation*}

\noindent Next, we define the linear functional
\begin{equation}
L_{A|W}X= Tr(A_{_W}X_{_W})\, \quad X \in {\cal S}^n
\end{equation}
observing that equivalently $L_{A|W}X=Tr(A_{_W}X)=Tr(AX_{_W})$.\\
Moreover, assuming $\lambda I_{_W}\le A_{_W} \le \Lambda I_{_W}$, in short $A_W \in [\lambda,\Lambda]$, for positive constants $\lambda$ and $\Lambda$, we have
\begin{equation}
\begin{split}
L_{A|W}X=& \,Tr(A_{_W}X^+_{_W}) - Tr(A_{_W}X^-_{_W}) \\
\le&\,\Lambda Tr(X^+_{_W}) - \lambda Tr(X^-_{_W}) \equiv {\cal P}^+_{\lambda,\Lambda|W}(X)\,
\end{split}
\end{equation}
and similarly
\begin{equation}
L_{_{A|W}}X \ge \lambda Tr(X^+_{_W}) - \Lambda Tr(X^-_{_W}) \equiv {\cal P}^-_{\lambda,\Lambda|W}(X)\,.
\end{equation}
As in the Introduction, the fully nonlinear operators ${\cal P}^+_{\lambda,\Lambda|W}$ and ${\cal P}^-_{\lambda,\Lambda|W}$ will be called the Pucci maximal and minimal operators on $W$, respectively, and the above shows  that
\begin{equation}\label{Pucci-W-ineq}
{\cal P}^-_{\lambda,\Lambda|W}(X) \le \inf_{A_W \in [\lambda,\Lambda]}L_{_{A|W}}X\le\sup_{A_W \in [\lambda,\Lambda]}L_{_{A|W}}X\le \,{\cal P}^+_{\lambda,\Lambda|W}(X). 
\end{equation}
On the other hand, for a fixed $X$ we construct $\tilde A \in {\cal S}^n$ such that $\tilde A=\Lambda I$, resp. $\tilde A=\lambda I$, on the linear subspace $V_+$, resp. $V_-$, spanned by the eigenvectors of $X_W$ corresponding to nonnegative, resp. negative, eigenvalues, so that 
\begin{equation}
{\cal P}^+_{\lambda,\Lambda|W}(X) = L_{_{\tilde A|W}}X \le \sup_{A_W \in [\lambda,\Lambda]}L_{_{A|W}}X
\end{equation}
and, reversing the role of $\lambda$ and $\Lambda$,
\begin{equation}
{\cal P}^-_{\lambda,\Lambda|W}(X) \ge \inf_{A_W \in [\lambda,\Lambda]}L_{_{A|W}}X.
\end{equation}
Hence, inequalities (\ref{Pucci-W-ineq}) can be restated more precisely as
\begin{equation}\label{Pucci-W-ineq2}
{\cal P}^-_{\lambda,\Lambda|W}(X) = \inf_{A_W \in [\lambda,\Lambda]}L_{_{A|W}}X\le\sup_{A_W \in [\lambda,\Lambda]}L_{_{A|W}}X= \,{\cal P}^+_{\lambda,\Lambda|W}(X)\,.
\end{equation}

\noindent From the above characterization it is not difficult to prove that the Pucci maximal and minimal operators on $W$ fulfill many properties of the standard Pucci operators, see for instance in Lemma 2.10 of \cite{CC}, which we list here below for convenience of the reader. 

\begin{lemma} \label{Pucci-W} The Pucci maximal and minimal operators restricted to $W$, respectively ${\cal P}^+_{\lambda,\Lambda|W}$ and ${\cal P}^-_{\lambda,\Lambda|W}$, as defined in $(\ref{Pucci-W-eq})$, are degenerate elliptic operators, uniformly elliptic if $p=n$, with the following properties:
\vskip0.25cm
\begin{tabular}{l l  l }
\hskip-0.75cm a)  &\hskip-0.5cm  ${\cal P}^-_{\lambda,\Lambda|W}(X) = -{\cal P}^+_{\lambda,\Lambda|W}(-X)$  \hfill {\rm(duality)}\\[1ex]
\hskip-0.75cm b)  &\hskip-0.5cm  ${\cal P}^\pm_{\lambda,\Lambda|W}(cX) = c{\cal P}^\pm_{\lambda,\Lambda|W}(X)$ \hbox{\rm if  } $c \ge 0$\quad   \hfill \hbox{\rm (positive homogeneity)} \\[1ex]
\hskip-0.75cm c) &\hskip-0.5cm ${\cal P}^+_{\lambda,\Lambda|W}(X) + {\cal P}^-_{\lambda,\Lambda|W}(Y) \le {\cal P}^+_{\lambda,\Lambda|W}(X+Y) \le {\cal P}^+_{\lambda,\Lambda|W}(X) + {\cal P}^+_{\lambda,\Lambda|W}(Y)$ \\[1ex]
& \hfill \hbox{\rm (subadditivity and reverse inequality)}\\ [1ex]
\hskip-0.75cm d)  &\hskip-0.5cm ${\cal P}^-_{\lambda,\Lambda|W}(X) + {\cal P}^-_{\lambda,\Lambda|W}(Y) \le {\cal P}^-_{\lambda,\Lambda|W}(X+Y) \le {\cal P}^+_{\lambda,\Lambda|W}(X) + {\cal P}^-_{\lambda,\Lambda|W}(Y)$\\ [1ex]
& \hfill \hbox{\rm (superadditivity and reverse inequality)}\\ [1ex]
\hskip-0.75cm e)  &\hskip-0.5cm ${\cal P}^-_{\lambda',\Lambda'|W}(X) \le {\cal P}^-_{\lambda,\Lambda|W}(X) \le {\cal P}^+_{\lambda,\Lambda|W}(X) \le {\cal P}^+_{\lambda',\Lambda'|W}(X)$ \ if \ $[\lambda,\Lambda]\subset [\lambda',\Lambda']$  \\ [1ex]
& \hfill {\rm (monotonicity with respect to the ellipticity interval)}\\
\end{tabular}\\
\end{lemma} 

\noindent Letting $W$ run over the Grassmannian ${\cal G}_p$ of all linear $p$-dimensional subspaces of $\mathbb R^n$, and taking the supremum and the infimum over ${\cal G}_p$, we will obtain the Pucci maximal and minimal operators of order $p$, respectively ${\cal P}^+_{\lambda,\Lambda|p}(X)$ and ${\cal P}^-_{\lambda,\Lambda|p}(X)$, as defined in (\ref{Pucci-p-eq}); see the Introduction. \\
To see this, for $W \in {\cal G}_p$ and $X \in {\cal S}^n$ we note that   
\begin{equation}\label{eq25}
\begin{split}
{\cal P}^+_{\lambda,\Lambda|W}(X) =&\sup_{A_W\in\left[\lambda,\Lambda\right]}Tr\left(A_{_W}X\right)\\
\le&\, \Lambda Tr(I_{_W}X^+) - \lambda Tr(I_{_W}X^-)\\
= &\, \Lambda Tr(I_{_W} OD^+O^T)- \lambda Tr(I_{_W} OD^-O^T)\\
= &\, \Lambda Tr(O^TI_{_W}OD^+)- \lambda Tr(O^TI_{_W} OD^-)\\
= &\, \Lambda Tr(I_{_{\widetilde W}}D^+)- \lambda Tr(I_{_{\widetilde W}} D^-).
\end{split}
\end{equation}
In the above, we have used the existence of an orthogonal matrix $O$, i.e. $OO^T=I=O^TO$, such that $O^TXO=D$ is diagonal, choosing $O$ in order that the eigenvalues $e_i(X)$ occur in nondecreasing order from the top to the bottom on the diagonal of $D$. \\
Moreover, let $B^{n-p+1},\dots,B^n$ be an orthonormal basis for $W$ and $B=(B^{n-p+1},\dots,B^n)$, then 
$$I_{_{\widetilde W}}= O^TI_{_W}O=O^TBB^TO=O^TB(O^TB)^T$$ 
is in turn a projection operator on a $p$-dimensional linear subspace $\widetilde W$, which is generated by the unit vectors $\tilde B^{n-p+1}= O^TB^{n-p+1},\ldots,\tilde B^{n}=O^TB^n$.  \\
From (\ref{eq25}), by linearity we have
\begin{equation}\label{D+D-}
{\cal P}^+_{\lambda,\Lambda|W}(X) \leq \, Tr(I_{_{\widetilde W}}(\Lambda D^+-\lambda D^-))\,,
\end{equation}
where
\begin{equation*}
 D^\pm = \left(\begin{array}{cccc}
 e_1^\pm(X) & 0 & \ldots & 0\\
0 &  e_{2}^\pm(X)& \ldots & 0\\
\vdots & \vdots &\ddots &\vdots\\
0 &\ldots &\ldots & e_n^\pm(X)
\end{array}\right)
\end{equation*}
and $e_i^\pm(X)=\max(\pm e_i(X),0)$.\\
Next, we estimate the right-hand side of (\ref{D+D-}), setting $\lambda_i=\Lambda e^+_i(X)-\lambda e^-_i(X)$ and $\tilde B_i^j$ the components of $\tilde B^j$, so that
\begin{equation*}
Tr(I_{_{\widetilde W}}(\Lambda D^+-\lambda D^-))=\sum_{i=1}^n\lambda_i\sum_{j=n-p+1}^n|\tilde B_i^j|^2\,.
\end{equation*}
Since the sequence of $\lambda_i$ is nondecreasing, from this we get
\begin{equation*}
\begin{split}
Tr(I_{_{\widetilde W}}(\Lambda D^+-\lambda D^-))&=\sum_{i=1}^{n-p}\lambda_i\sum_{j=n-p+1}^n|\tilde B^j_i|^2+\sum_{i=n-p+1}^{n}\lambda_i\left(\sum_{j=n-p+1}^n|\tilde B_i^j|^2-1\right)\\
&+\sum_{i=n-p+1}^{n}\lambda_i\\
&\le\lambda_{n-p}\sum_{j=n-p+1}^{n}\left(\sum_{i=1}^n |\tilde B^j_i|^2-1\right)+\sum_{i=n-p+1}^{n}\lambda_i\\
&=\sum_{i=n-p+1}^n\left(\Lambda e^+_i(X)-\lambda e^-_i(X)\right).
\end{split}
\end{equation*}
Combining this inequality with the above (\ref{D+D-}), we get
\begin{equation}\label{bound}
{\cal P}^+_{\lambda,\Lambda|W}(X) \leq \,\sum_{i=n-p+1}^n\left(\Lambda e^+_i(X)-\lambda e^-_i(X)\right).
\end{equation}
On the other side, if $W$ is the linear subspace $W_0$, mapped by the orthogonal transformation $O^T$, such that $O^TXO=D$, into the linear subspace $\widetilde W_0=\{0\}^{n-p}\times\mathbb R^p$ spanned by $E^{n-p+1},\dots, E^n$, then inequality in (\ref{D+D-}) is achieved. Indeed, since
\begin{equation*}
D_{_{\widetilde W_0}}=\left(\begin{array}{cccccc}
0      & \ 0      & \  \ldots & \ldots&\ldots & 0\\
0      & \ \ddots & \  \ldots & \ldots&\ldots & \vdots\\
\vdots & \ \ldots & \ 0 &\ldots &\ldots & \vdots\\
\vdots & \ \ldots & \ \ldots & e_{n-p+1}(X) &\ldots &\vdots\\
\      &   \      &    \      &   \         & \     & \   \\
\vdots & \ \ldots & \  \ldots &\ldots & \ddots & \vdots\\
0      & \ \ldots & \ \ldots &\ldots &  \ldots & e_n(X)
\end{array}\right),
\end{equation*}

\noindent and

\begin{equation*}
\begin{split}
{\cal P}_{\lambda,\Lambda|W_0}^+(X)&=\sup_{A_{W_0}\in[\lambda,\Lambda]}Tr(AX_{_{W_0}})=\sup_{A_{\widetilde W_0}\in[\lambda,\Lambda]}Tr(OAO^TX_{_{W_0}})\\
&=\sup_{A_{\widetilde W_0}\in[\lambda,\Lambda]}Tr\left((OAO^T)_{_{W_0}}X\right)\\
&=\sup_{A_{\widetilde W_0}\in[\lambda,\Lambda]}Tr\left(O^T(OAO^T)_{_{W_0}}OD\right)\\
&=\sup_{A_{\widetilde W_0}\in[\lambda,\Lambda]}Tr\left(A_{_{\widetilde W_0}}D\right)={\cal P}_{\lambda,\Lambda|\widetilde W_0}^+(D)
\end{split}
\end{equation*} 
we conclude that
\begin{equation}\label{maximality}
{\cal P}_{\lambda,\Lambda|W_0}^+(D)=\sum_{i=n-p+1}^n\left(\Lambda e^+_i(X)-\lambda e^-_i(X)\right).
\end{equation}
As a consequence we obtain the following representation for the Pucci maximal operator of order $p$:
\begin{equation}
\begin{split}
\sup_{W \in {\cal G}_p}{\cal P}^+_{\lambda,\Lambda|W}(X)= \sum_{i=n-p+1}^n\left(\Lambda e^+_i(X)-\lambda e^-_i(X)\right) \equiv \, {\cal P}^+_{\lambda,\Lambda|p}(X).
\end{split}
\end{equation}
\noindent A similar computation can be carried out for the minimal Pucci operator, showing that
\begin{equation}
\begin{split}
\inf_{W \in {\cal G}_p}{\cal P}^-_{\lambda,\Lambda|W}(X) = \sum_{i=1}^p \left(\Lambda e^+_i(X)-\lambda e^-_i(X)\right).
\end{split}
\end{equation}
Therefore
\begin{equation}\label{Pucci-p-ineq}
{\cal P}^-_{\lambda,\Lambda|p}(X) = \inf_{W \in {\cal G}}{\cal P}^-_{\lambda,\Lambda|W}(X)\le\sup_{W \in {\cal G}}{\cal P}^+_{\lambda,\Lambda|W}(X)= \,{\cal P}^+_{\lambda,\Lambda|W}(X)\,.
\end{equation}
From (\ref{Pucci-W-ineq2}) and the above characterization we also obtain the following representation for the Pucci maximal and minimal operators of order $p$:
\begin{equation}\label{sup/inf:representation}
\begin{split}
{\cal P}^+_{\lambda,\Lambda|p}(X) = \sup_{W \in {\cal G}_p}\,\sup_{A_W \in [\lambda,\Lambda]}L_{_{A|W}}X,\\
{\cal P}^-_{\lambda,\Lambda|p}(X) = \inf_{W \in {\cal G}_p}\,\inf_{A_W \in [\lambda,\Lambda]}L_{_{A|W}}X.
\end{split}
\end{equation}
It is also worth to remark that from Lemma \ref{Pucci-W} and the above characterization, we can state for degenerate elliptic Pucci operators of order $p$ the analogous of Lemma \ref{Pucci-W}.

\begin{lemma} \label{Pucci-p} The Pucci maximal and minimal operators of order $p$, respectively ${\cal P}^+_{\lambda,\Lambda|p}$ and ${\cal P}^-_{\lambda,\Lambda|p}$, as defined in $(\ref{Pucci-p-eq})$, are degenerate elliptic operators, uniformly elliptic if $p=n$, satisfying the properties {\rm a) $\div$ e)} of {\rm Lemma \ref{Pucci-W}} with $W$ replaced by $p$.
\end{lemma}

\noindent We end this section noticing that the operators ${\cal P}^\pm_{\lambda,\Lambda|p}$, which are degenerate elliptic by Lemma \ref{Pucci-p}, are not uniformly elliptic for $p<n$, as the following simple example shows in the case of $\lambda =\Lambda=1$, $p=1$, $n=2$. In fact, if 
\begin{equation*}
X= \left(\begin{array}{cc}
1      & 0  \\
0      & 0 
\end{array}\right), \ P= \left(\begin{array}{cc}
0      & 0  \\
0      & 1
\end{array}\right) \ge \left(\begin{array}{cc}
0      & 0  \\
0      & 0
\end{array}\right),
\end{equation*}
then
$$
e_2(X+P)-e_2(X) = 0 < \varepsilon = \varepsilon Tr(P) \quad \forall\,\varepsilon>0,
$$
thereby contradicting uniform ellipticity.

\noindent Finally, we remark that, in the same way all uniformly elliptic operators with ellipticity constants $\lambda>0$ and $\Lambda\ge \lambda$ are included between the minimal and maximal Pucci operators ${\cal M}^-_{\lambda,\Lambda}={\cal P}^-_{\lambda,\Lambda|n}$ and ${\cal M}^+_{\lambda,\Lambda}={\cal P}^+_{\lambda,\Lambda|n}$, we can include them in a larger range between the minimal and maximal degenerate elliptic Pucci operators ${\cal P}^-_{\frac np\,\lambda,\frac np\,\Lambda|p}$ and ${\cal P}^+_{\frac np\lambda,\frac np\Lambda|p}$, with $p<n$, here introduced. \\
In fact, the following inclusions hold true for all $X \in {\cal S}^n$:
\begin{equation}\label{inclusions}
{\cal P}^-_{\frac np\lambda,\frac np\Lambda|p}(X) \le {\cal M}^-_{\lambda,\Lambda}(X)={\cal P}^-_{\lambda,\Lambda|n}(X) \le {\cal M}^+_{\lambda,\Lambda}(X) \le {\cal P}^+_{\frac np\lambda,\frac np\Lambda|p}(X). 
\end{equation} 
To show this, note that, since the eigenvalues $e_i(X)$ are arranged in nondecreasing order, then the sequence of numbers $\Lambda e^+_i(X)-\lambda e^-_i(X)$ is in turn nondecreasing, so that 
\begin{align*}
  & \sum_{i=1}^n\left(\Lambda e^+_i(X)-\lambda e^-_i(X) \right) \\
= & \sum_{i=1}^{n-p}\left(\Lambda e^+_i(X)-\lambda e^-_i(X) \right) + \sum_{i=n-p+1}^{n}\left(\Lambda e^+_i(X)-\lambda e^-_i(X) \right)\\
\le & \left(\frac{n-p}{p} + 1\right) \sum_{i=n-p+1}^{n}\left(\Lambda e^+_i(X)- \lambda e^-_i(X) \right)\\
= &  \sum_{i=n-p+1}^{n}\left( \frac np\,\Lambda e^+_i(X)- \frac np\,\lambda e^-_i(X) \right)
\end{align*}
which proves the last inequality in (\ref{inclusions}), whereas the first one can be obtained by duality.

\section{Viscosity solutions of degenerate \\ Pucci equations}\label{Viscosity}

\noindent Denote by $LSC(\Omega)$, resp. $USC(\Omega)$, the space of lower, resp. upper, semicontinuous function $u:\Omega\mapsto(-\infty,+\infty]$, resp. $u:\Omega\mapsto[-\infty,+\infty)$, on the domain $\Omega\subseteq\mathbb R^n$. Let $f(x)$ be a continuous function in $\Omega$. We say that $u\in LSC(\Omega)$ is a \emph{viscosity supersolution} of the equation
\begin{equation}\label{eqpr1}
F(x,u(x),Du(x),D^2u(x))=f(x)\quad{\text{in $\Omega$}}
\end{equation}
if for any $x_0\in\Omega$ such that $u(x_0)<+\infty$ the inequality
\begin{equation}\label{eqpr2}
F(x_0,u(x_0),D\phi(x_0),D^2\phi(x_0))\leq f(x_0)
\end{equation}
holds true for any test function $\phi\in C^2(\Omega)$ such that $u-\phi$ attains a local minimum at $x_0$. In a symmetric way, choosing $x_0$ in such a way $u(x_0)>-\infty$ and using test functions $\phi$ such that $u-\phi$ has a local maximum, we obtain the definition of \emph{viscosity subsolution} for $u\in USC(\Omega)$ by reversing the inequality (\ref{eqpr2}). As usual, eventually replacing $\phi$ by $\phi\pm|x-x_0|^4$, we may assume that the local maximum and minimum are strict.\\ 
Henceforth we also say that $F(x,u(x),Du(x),D^2u(x))\leq f(x)$, respectively 
$F(x,u(x),Du(x),D^2u(x))\geq f(x)$,  is fulfilled in the viscosity sense whenever $u$ is a viscosity supersolution, resp. subsolution, of equation (\ref{eqpr1}). Finally we refer to a \emph{viscosity solution} $u$ of (\ref{eqpr1}) as a continuous function in $\Omega$ satisfying both the previous inequalities.\\
An equivalent way of defining viscosity solutions involves the notion of second order \emph{semijets}: the \emph{subjet} $J^{2,-}u(x)$, respectively \emph{superjet} $J^{2,+}u(x)$, of $u$ at $x\in\Omega$ is the convex set of all pairs $(\xi,X)\in\mathbb R^n\times{\cal S}^n$ such that 
$$\ u(y)\geq u(x)+\left\langle \xi,y-x\right\rangle+\frac{1}{2}\left\langle X(y-x),y-x\right\rangle+o(|y-x|^2)\quad\text{as } y\to x,$$
respectively
$$\ u(y)\leq u(x)+\left\langle \xi,y-x\right\rangle+\frac{1}{2}\left\langle X(y-x),y-x\right\rangle+o(|y-x|^2)\quad\text{as } y\to x.$$
Note that $J^{2,-}u(x)=-J^{2,+}(-u)(x)$ and if $u$ is twice differentiable at $x\in\Omega$ then 
\begin{align*} J^{2,+}u(x)=&\left\{(Du(x),X):\;X\geq D^2u(x)\right\},\\ \ J^{2,-}u(x)=& \left\{(Du(x),X):\;X\leq D^2u(x)\right\}.\end{align*}
A function $u\in LSC(\Omega)$ ($USC(\Omega)$) is a viscosity supersolution, resp. a subsolution, of (\ref{eqpr1}) if
\begin{equation}\label{eqpr3}
\begin{split}
&F(x,u(x),\xi,X)\leq f(x)\quad\forall (\xi,X)\in J^{2,-}u(x)\,, \ \hbox{\rm resp.}\\
&F(x,u(x),\xi,X)\geq f(x)\quad\forall (\xi,X)\in J^{2,+}u(x)
\end{split}
\end{equation}
The closures of the semijets are defined as 
\begin{equation*}
\begin{split}
&\overline J^{2,\pm}u(x)=\left\{(\xi,X)\in\mathbb R^n\times{\cal S}^n:\,\exists (x_\alpha,\xi_\alpha,X_\alpha)\in\Omega\times\mathbb R^n\times{\cal S}^n, \right.\\
&\left. (\xi_\alpha,X_\alpha)\in J^{2,\pm}u(x_\alpha)\text{ and } (x_\alpha,u(x_\alpha),\xi_\alpha,X_\alpha)\to(x,u(x),\xi,X)\;\text{as}\;\alpha\to\infty\right\}.
\end{split}
\end{equation*}
Inequalities (\ref{eqpr3}) remains true  if $(\xi,X)\in \overline J^{2,\pm}u(x)$ by continuity.

\noindent {Viscosity solutions are stable with respect to upper and lower semicontinuous envelope: let ${\cal F}$ be a family of subsolutions of (\ref{eqpr1}), $w(x)=\sup\left\{u(x):\,u\in{\cal F}\right\}$ and
\begin{equation*}
\begin{split}
w^*(x)&=\lim_{r\to0^+}\sup\left\{w(y):\,\text{$y\in\Omega$ and $|y-x|\leq r$}\right\}\\
&=\inf\left\{v(x):\,\text{$v\in USC(\Omega)$ and $v\geq w$ in $\Omega$}\right\}
\end{split}
\end{equation*}
the upper semicontinuous regularization of $w$; if $w^*<+\infty$ then $w^*$ is in turn a subsolution of (\ref{eqpr1}). An analogous result involving the lower semicontinuous regularization bounded below  holds for supersolutions. See \cite[Lemma 4.2]{USER} for further details.\\
In Proposition \ref{potential} we will also use another stability property, that the limit of a non-decreasing sequence of continuous viscosity supersolution of (\ref{eqpr1}) is a supersolution of the same equation. For a proof we refer to \cite[Lemma 2.1]{AGV}, but it is a well known result in the literature on viscosity solutions.

\noindent In order to compare viscosity subsolutions and supersolutions we will use the following lemma, which provides a maximal equation for their difference.
 
 \begin{lemma}\label{lemdifference}
Let $u$ and $v$ be respectively  viscosity  subsolutions and supersolutions in $\Omega$ of 
\begin{equation*}
F(u,Du,D^2u)= f(x)\;\;\text{and}\;\;F(v,Dv,D^2v)= g(x)
\end{equation*}
where $f,g\in C(\Omega)$ and $F$ satisfies (\ref{SC}). Then the difference $w=u-v$ is a viscosity solution of the maximal differential inequality
$${{\cal P}^+_{\lambda,\Lambda|p}}(D^2w)+b|Dw|-cw\geq f(x)-g(x)\quad\text{in $\Omega$}.$$
\end{lemma}
\begin{proof}
Let $x_0\in\Omega$, $\phi\in C^2(\Omega)$ such that $w(x_0)>-\infty$ and $u-v-\phi$ has a strict local  maximum at $x_0$. Choose $r>0$ such that
\begin{equation}\label{eq1lem1}
(u-v-\phi)(x_0)>(u-v-\phi)(x)\quad\forall x\in \overline B_r(x_0)\backslash\left\{x_0\right\}.
\end{equation}
The proof follows the lines of \cite[Chapter 3]{USER}. For $\alpha>0$ let 
$$(u-\phi)(x_\alpha)-v(y_\alpha)-\frac{\alpha}{2}|x_\alpha-y_\alpha|^2=\max_{\overline B_r(x_0)\times \overline B_r(x_0)}\left((u-\phi)(x)-v(y)-\frac{\alpha}{2}|x-y|^2\right).$$
In view of \cite[Lemma 3.1]{USER} we may assume, up to a subsequence, that $(x_\alpha,y_\alpha)\to(\tilde{x},\tilde{x})$, with $\tilde{x}\in\overline B_r(x_0)$ and 
\begin{equation*}
\begin{split}
(u-\phi-v)(\tilde{x})&=\lim_{\alpha\to+\infty}\left((u-\phi)(x_\alpha)-v(y_\alpha)-\frac{\alpha}{2}|x_\alpha-y_\alpha|^2\right)\\
&\geq\max_{\overline B_r(x_0)}\left(u-\phi-v\right)(x)=(u-\phi-v)(x_0).
\end{split}
\end{equation*}
From (\ref{eq1lem1}) we deduce that $\tilde x=x_0$ as well as $u(x_\alpha)\to u(x_0)$ and $v(y_\alpha)\to v(x_0)$ as $\alpha\to+\infty$. Using \cite[Theorem 3.2]{USER} there exist $X_\alpha$, $Y_\alpha\in {\cal S}^n$ such that
\begin{equation}\label{eq2lem1}
\left(\alpha(x_\alpha-y_\alpha),X_\alpha\right)\in\overline J^{2,+}(u-\phi)(x_\alpha)
\end{equation}
\begin{equation}\label{eq3lem1}
\left(\alpha(x_\alpha-y_\alpha),Y_\alpha\right)\in\overline J^{2,-} v(y_\alpha)
\end{equation}
and
\begin{equation}\label{eq4lem1}
\left( 
\begin{array}{cc}
	X_\alpha & 0\\
	0 & -Y_\alpha
\end{array}\right)\leq
3\alpha\left( 
\begin{array}{cc}
	I & -I\\
	-I & I
\end{array}\right)\,
\end{equation}
which implies  $X_\alpha\leq Y_\alpha$. As a consequence of (\ref{eq2lem1}) we have
\begin{equation}\label{eq5lem1}
\left(\alpha(x_\alpha-y_\alpha)+D\phi(x_\alpha),X_\alpha+D^2\phi(x_\alpha)\right)\in\overline J^{2,+} u(x_\alpha)
\end{equation}
and by (\ref{SC})
\begin{equation*}
\begin{split}
f(x_\alpha)-g(y_\alpha)&\leq F\left(u(x_\alpha),\alpha(x_\alpha-y_\alpha)+D\phi(x_\alpha),X_\alpha+D^2\phi(x_\alpha)\right)\\
&-F\left(v(y_\alpha),\alpha(x_\alpha-y_\alpha),Y_\alpha\right)\\
&\leq{{\cal P}^+_{\lambda,\Lambda|p}}\left(X_\alpha-Y_\alpha+D^2\phi(x_\alpha)\right)+b|D\phi(x_\alpha)|-c(u(x_\alpha)-v(y_\alpha))\\
&\leq{{\cal P}^+_{\lambda,\Lambda|p}}\left(D^2\phi(x_\alpha)\right)+b|D\phi(x_\alpha)|-c(u(x_\alpha)-v(y_\alpha)).
\end{split}
\end{equation*}
Letting $\alpha\to+\infty$ we conclude that
$$(f-g)(x_0)\leq{{\cal P}^+_{\lambda,\Lambda|p}}\left(D^2\phi(x_0)\right)+b|D\phi(x_0)|-c(u(x_0)-v(x_0)).$$
\end{proof}

\noindent Let 
$\Omega$ be a domain $\mathbb R^n$ and $E$ be a closed subset of $\overline\Omega$ with empty interior. 
Following Harvey and Lawson \cite[Sections 3 and 6]{HL2}, we define the upper semicontinuous extension across $E$ of a function $u \in USC(\Omega\backslash E)$, bounded above, setting
$$U(x)\equiv\limsup_{\substack{y\to x\\ y\notin E}}u(y)= \lim_{\varepsilon \to 0^+}\sup_{\substack{y \in 
B_\varepsilon(x) \\ y\in \Omega\backslash E}} u(y)\,.$$
Note that $U(x)$ is the upper semicontinuous regularization $\tilde u^*$ of the function 
\begin{equation*}
\tilde u(x)=\begin{cases}
u(x) & \text{outside $E$}\\
-\infty & \text{on $E$.}
\end{cases}
\end{equation*}
In the same manner, the lower semicontinuous extension across $E$ of a function $v \in LSC(\Omega\backslash E)$, bounded below, is defined setting
$$V(x)\equiv\liminf_{\substack{y\to x\\ y\notin E}}v(y)= \lim_{\varepsilon \to 0^+}\inf_{\substack{y \in 
B_\varepsilon(x) \\ y\in \Omega\backslash E}} v(y)\,.$$
The function $V\in LSC(\Omega)$ and coincide with the lower semicontinuous regularization $\tilde v_*$ of 
\begin{equation*}
\tilde v(x)=\begin{cases}
v(x) & \text{outside $E$}\\
+\infty & \text{on $E$.}\end{cases}
\end{equation*}

\section{A priori bounds}\label{A priori bound}

\noindent In this Section we show above estimates for viscosity subsolutions of the maximal equation (\ref{eq5intr}).

\begin{proposition}\label{MPsmall}
Let  $\Omega \subseteq B_\delta$ be a domain of $\mathbb R^n$  and suppose that  $u\in USC(\Omega)$ is a viscosity subsolution of 
$${{\cal P}^+_{\lambda,\Lambda|p}}(D^2u(x))+b|Du(x)|=f(x)\quad\text{ in $\Omega$}$$
with $f\in C(\Omega)$. If 
\begin{equation}\label{bdelta}
b\delta<{\lambda}p
\end{equation}
then
\begin{equation}\label{MP}
\sup_\Omega u\leq\limsup_{y\to\partial\Omega}u(y)+C\left\|f^-\right\|_\infty
\end{equation}
where  $C$ is a constant depending only on {$\lambda$}, $p$, $b$ and $\delta$.
\end{proposition}
\begin{proof} Suppose firstly that $f\equiv0$ and $\limsup_{y\to\partial\Omega}u(y)\leq0$\,:
we claim that $u\leq0$ in $\Omega$. \\
By contradiction, suppose $u(\overline x)>0$ for $\overline x\in\Omega$. Take a positive number $\varepsilon<u(\overline x)\delta^{-2}$ and consider $\phi(x)=-\varepsilon |x|^2$. The difference $u-\phi$ attains its maximum at a point $x_\varepsilon$ inside $\Omega$ considering that $\forall x\in\partial\Omega$ 
$$\limsup_{y\to x}(u-\phi)(y)\leq\varepsilon  |x|^2\leq\varepsilon  \delta^2<u(\overline x)\leq (u-\phi)(\overline x).$$ 
Using $\phi$ as test function at $x_\varepsilon$ we get the contradiction
\begin{equation*}
0\leq {\cal P}^+_{\lambda,\Lambda|p}(D^2\phi(x_\varepsilon))+b|D\phi(x_\varepsilon)|=-2\varepsilon\lambda p+2b\varepsilon|x_\varepsilon|<0,
\end{equation*}
which proves what claimed: $u\leq0$ in $\Omega$.
\noindent For the general case,  note that the smooth function
\begin{equation}\label{v}
v(x)=\gamma\left(\delta^2-\left|x\right|^2\right)+\limsup_{y\to\partial\Omega}u(y),
\end{equation}
with $\gamma=\frac{\left\|f^-\right\|_\infty}{2({\lambda}p-b\delta)}$,
 is a classical supersolution in $\Omega$ of the equation
$${{\cal P}^+_{\lambda,\Lambda|p}}(D^2v(x))+b\left|Dv(x)\right|=-\left\|f^-\right\|_\infty,$$
while the difference $w=u-v$ satisfies in $\Omega$ the inequalities
\begin{equation*}
\begin{split}
{{\cal P}^+_{\lambda,\Lambda|p}}(D^2w(x))+b\left|Dw(x)\right|&\geq {{\cal P}^+_{\lambda,\Lambda|p}}(D^2u(x))+b\left|Du(x)\right|\\&-\,{{\cal P}^+_{\lambda,\Lambda|p}}(D^2v(x))-b\left|Dv(x)\right|\\
&\geq f(x)+\left\|f^-\right\|_\infty\geq0
\end{split}
\end{equation*}
in the viscosity sense. \\
Since $\displaystyle\limsup_{x\to\partial\Omega}w(x)\leq0$, we deduce from the previous case that $\displaystyle\sup_{x\in\Omega}w(x)\leq0$. Hence we obtain estimate (\ref{MP}) with $C=\frac{\delta^2}{2({\lambda}p-b\delta)}$.
\end{proof}

\begin{remark}
\rm In the case $p=n$,  the degenerate elliptic operator of the above Proposition is in fact uniformly elliptic, since second order term is the Pucci maximal operator ${\cal M}^+_{\lambda,\Lambda}\equiv {\cal P}^+_{\lambda,\Lambda|n}$. It is worth to recall that in this case estimate (\ref{MP}) holds true without any restriction on the size of $b\delta$.
\end{remark}

\begin{remark}
\rm The validity of the Maximum Principle fails to hold, when $p<n$, for large domains or large gradient coefficients and the bound  for the product $b\delta$ of Proposition \ref{MPsmall} is optimal, as the counterexample here below shows.
\end{remark}

\noindent \textbf{Counterexample.} For a sufficiently small positive $\varepsilon$, say $\varepsilon <\frac \pi{6}$, let use define the radial function ($r=|x|$) in $B_\delta$, where $\delta=\frac{\pi}{2} + \frac\varepsilon2$:
\begin{equation}\label{counterex}
u(x)= \left\{\begin{array}{ll}
\cos{\frac\varepsilon2} & {\rm if} \ r \le \frac{\pi}{2} - \frac\varepsilon2\\
\sin |x| & {\rm if} \ \frac{\pi}{2} - \frac\varepsilon2 \le r \le \frac{\pi}{2} + \frac\varepsilon2
\end{array}\right.
\end{equation}
The eigenvalues of $D^2u(x)$ are
\begin{equation}
e_1(D^2u(x))=-\sin r,\;\;e_2(D^2u(x))=\ldots=e_n(D^2u(x))=\frac{\cos r}{r},
\end{equation}
in the annular domain $r \in (\frac{\pi}{2} - \frac\varepsilon2, \frac{\pi}{2} + \frac\varepsilon2)$, where, choosing
\begin{equation}\label{eqcoefficient}
b = \frac{\lambda p}{\delta-\varepsilon}\,,
\end{equation}
 we get
\begin{equation}
\begin{split}
&\quad\;{\cal P}^+_{\lambda,\Lambda|p}(D^2u(x))+b|Du(x)|\\&\equiv \Lambda\sum_{i=n-p+1}^ne^+_{i}(D^2u(x))-\lambda\sum_{i=n-p+1}^ne^-_{i}(D^2u(x))+b|Du(x)|\\
&=\Lambda p\,\frac{(\cos r)^+}{r}-\lambda p\,\frac{(\cos r)^-}{r}+b\,|\cos r|\\
&\geq\frac pr\left(\Lambda(\cos r)^+-\lambda(\cos r)^-+\lambda|\cos r|\right)\geq0\,.
\end{split}
\end{equation}
By viscosity, it follows that $u(x)$ is a subsolution of equation  ${\cal P}^+_p(D^2u)+b|Du| = 0$ in all $B_\delta$, but $\max_{\overline B_\delta}u=1 > \cos{\frac\varepsilon2} = \max_{\partial B_\delta}u$, so contradicting the Maximum Principle. By (\ref{eqcoefficient}) this occurs as soon as $b\delta >\lambda p$.

\begin{remark}
It is worth to point out that the validity of the Maximum Principle depends in essential way on the diameter of the domain $\Omega$, once fixed $b$, and in general it is not possible to relax this dependence by requiring that the measure $|\Omega|$ has to be sufficiently small, as in the case of narrow domains. To see this, it is sufficient to consider in the previous counterexample the restriction of $u$ to the spherical shell  $\Omega_\varepsilon=\left\{x\in\mathbb R^n:\,\frac{\pi}{2} - \frac\varepsilon2 \leq|x|\leq \frac{\pi}{2} + \frac\varepsilon2 \right\}$, observing that $|\Omega_\varepsilon|\to0$ as $\varepsilon\to0^+$, while 
$\max_{\overline\Omega_\varepsilon}u>\max_{\partial \Omega_\varepsilon}u.$
\end{remark}

\noindent The above argument clarifies that, differently from the uniformly elliptic case, the size of the gradient term has to be suitably small in order that the Maximum Principle continue to hold. The restriction (\ref{bdelta}) of Proposition \ref{MPsmall} will be not necessary if we deal with elliptic equations coercive with respect to zero order term, namely
\begin{equation}\label{eqmp1}
{\cal P}^+_{\lambda,\Lambda|p}(D^2u)+b|Du|-cu=f(x)
\end{equation}
where $c$ is a positive constant, see \cite[Theorem 3.3]{USER}.\\
To see how this condition does work, suppose $u(x)$ to be a subsolution of equation (\ref{eqmp1}) with $f(x)=0$. 
If $u(x)$ would have a positive maximum $u(x_0)=M$ at $x_0 \in \Omega$, then the constant function $u=M$ should be a test function at $x_0$ and therefore
\begin{equation}\label{eqmp1ter}
0 = {\cal P}^+_{\lambda,\Lambda|p}(D^2M)+b|DM| \ge cM >0\,
\end{equation}
a contradiction that shows $u \le 0$ in $\Omega$, and Maximum Principle holds true.\\
Following the proof of Proposition \ref{MPsmall}, we also establish here below the analogous of estimate (\ref{MP}) in the case $c>0$ without any condition on the size of the gradient term and of the domain. 

\begin{proposition}\label{MPc}
Let $\Omega \subseteq B_\delta$ be a domain of $\mathbb R^n$ and let 
$u\in USC(\Omega)$ be a viscosity subsolution of equation $(\ref{eqmp1})$ in $\Omega$ with $f\in C(\Omega)$. If $c>0$, 
then
\begin{equation}\label{MP2}
\sup_\Omega u\leq\limsup_{y\to\partial\Omega}u^+(y)+C\left\|f^-\right\|_\infty
\end{equation}
where  $C$ is a constant depending only on {$\lambda$}, $p$, $b$, $c$ and $\delta$.
\end{proposition}
\begin{proof} The proof follows the same lines of the second part of Proposition \ref{MPsmall} (general case), but here we have to consider the slightly modified function
\begin{equation}\label{v-c}
v(x)=\gamma\left(\delta^2+\frac{2}{c}\,(\lambda p-b\delta)^-+\varepsilon-\left|x\right|^2\right)+\limsup_{y\to\partial\Omega}u^+(y),
\end{equation}
for any positive number $\varepsilon$ and $\gamma=\frac{\left\|f^-\right\|_\infty}{2(\lambda p-b\delta)^++c\varepsilon}$, in order to get
$${{\cal P}^+_{\lambda,\Lambda|p}}(D^2v(x))+b\left|Dv(x)\right|-cv(x)\le -\left\|f^-\right\|_\infty.$$
Introducing the function $w=u-v$ and using the Maximum Principle deduced from (\ref{eqmp1ter}), as in the last part of the proof of Proposition \ref{MPsmall}, inequality (\ref{MP2}) follows with $C=\frac{\delta^2+\frac2c(\lambda p-b\delta)^-+\varepsilon}{2(\lambda p-b\delta)^++c\varepsilon}$.
\end{proof} 

\section{Extended a priori bounds}\label{Extended a priori bound}

\noindent As stated in the Introduction we are concerned with an extended version of the Maximum Principle and more generally of  a priori estimates (\ref{MP}),(\ref{MP2}).\\
{Let ${\cal M}^+_1(E)$ be the set of positive Borel measure on the compact set $E\subset B_d$, normalized by $\mu(E)=1$.\\ Let $\Phi_\alpha(x)=|x|^{-\alpha}$ for $0<\alpha<n$ and 
$\Phi_0=\log\left(\frac{2d}{|x|}\right)$ be the Riesz kernels. \\ 
The $\alpha$-Riesz potential of $\mu\in{\cal M}^+_1(E)$ is defined as
\begin{equation}\label{4eq1}
V_\alpha^\mu(x)=\Phi_\alpha\ast\mu=\int_E\Phi_\alpha(x-y)d\mu(y)\in LSC(\mathbb R^n)\cap C^\infty(\mathbb R^n\backslash E).
\end{equation} 
We also denote by
$$
V_\alpha(E)=\inf_{\substack{\mu\in{\cal M}^+_1(E)}}\int_E V_\alpha^\mu(x)\,d\mu(x)
$$
the $\alpha$-equilibrium value on $E$ and by
$$
{\rm Cap}_\alpha(E)=\left\{\begin{array}{ll}
V_\alpha(E)^{-1} & {\rm if} \ \ \alpha >0,\\ \\
e^{-V_\alpha(E)} & {\rm if} \ \ \alpha =0
\end{array}\right.
$$
the $\alpha$-Riesz capacity of $E$.\\
 
\noindent For arbitrary $E$, not necessarily compact, the \emph{inner} $\alpha$-capacity  $\underline {\rm Cap}_{\,\alpha}(E)$ and the \emph{outer} $\alpha$-capacity $\overline {\rm Cap}_\alpha(E)$ are defined by 
$$\underline {\rm Cap}_{\,\alpha}(E)=\sup_{\substack{K \subset E\\ K \ {\rm compact}}}{\rm Cap}_\alpha(K),\quad \overline {\rm Cap}_\alpha(E)=\inf_{\substack{A \supset E\\ A \ {\rm open}}}\underline{\rm Cap}_{\,\alpha}(A).$$
For any $E$ one has $\underline {\rm Cap}_{\,\alpha}(E)\leq\overline {\rm Cap}_\alpha(E)$. If the equality holds true we say that  $E$ is  $\alpha$-\emph{capacitable} and  ${\rm Cap}_\alpha(E):=\underline {\rm Cap}_{\,\alpha}(E)=\overline {\rm Cap}_\alpha(E)$. Borel sets, in particular $F_\sigma$-sets, are capacitable \cite[Theorem 2.8]{L}. 

\begin{proposition}\label{potential} 
Let $\alpha^*=\frac{\lambda}{\Lambda}(p-1)-1\ge0$, $E$ be a Borel subset of $B_d$ with ${\rm Cap}_\alpha(E)=0$ and $x_0\in B_d\backslash E$. There exists a  non-negative function $v\in LSC(B_d)$ and a non-negative constant $K=K_{\lambda,\Lambda|p}(\alpha;b)$, defined for $\alpha \in [0,\alpha^*]$ if $b=0$,  for $\alpha \in [0,\alpha^*)$ if $\alpha^*>0$ and $b >0$, such that $v(x)=+\infty$ on $E$ and 
\begin{equation}\label{potenziale-supersoln}
{\cal P}^+_{\lambda,\Lambda|p}(D^2v(x))+b|Dv(x)|\leq K\quad\text{in $B_d$}.
\end{equation}
Moreover $v(x)<+\infty$ for any $x\in B_d\backslash E$  if $E$ is compact and $v(x_0)<+\infty$ if $E$ is an $F_\sigma$-set.
\end{proposition}

\begin{proof}
Firstly assume $E$ to be compact set such that ${\rm Cap}_\alpha(E)=0$. \\
In view of  \cite[Theorem 3.1]{L} there exists a unit positive measure $\mu$ for which $V^\mu_\alpha$
blows up on $E$ and is finite outside $E$:
\begin{equation}\label{4eq2}
V^\mu_\alpha\equiv+\infty\quad\text{on $E$},\quad V^\mu_\alpha<+\infty\quad{\text{in $\mathbb R^n\backslash E$}}.
\end{equation}
Moreover, 
\begin{equation}\label{4eq3}
V^\mu_\alpha \ge 0 \ \  \hbox{\rm in \ }
\left\{\begin{array}{lll}
\mathbb R^n & \hbox{\rm for} \ \alpha >0 \\
B_d & \hbox{\rm for} \ \alpha =0.\\
\end{array}\right.
\end{equation}
\noindent In the case $b=0$, for $x \not\in E$, differentiating under the integral and using the representation of ${\cal P}^+_{\lambda,\Lambda|p}$ as a supremum (\ref{sup/inf:representation}), we have
\begin{equation*}
\begin{split}
{\cal P}^+_{\lambda,\Lambda|p}(D^2V_\alpha^\mu(x))&= {\cal P}^+_{\lambda,\Lambda|p}\left(D^2 \int_E\Phi_\alpha(x-y)d\mu(y)\right)\\
&\leq \int_E {\cal P}^+_{\lambda,\Lambda|p}\left(D^2\Phi_\alpha(x-y)\right)\,d\mu(y)\\
&=\int_E\frac{(\alpha+\delta_{0,\alpha})\left(\Lambda (\alpha+1)-\lambda(p-1)\right)}{|x-y|^{\alpha+2}}\,d\mu(y),
\end{split}
\end{equation*}
where $\delta_{0,\alpha}=1$ if $\alpha=0$, $\delta_{0,\alpha}=0$ otherwise, and the right-hand side is non-positive by assumption $0\le \alpha \le \frac{\lambda}{\Lambda}(p-1)-1$.\\

\noindent 
Supposing $b>0$ and $0 \le \alpha <\frac{\lambda}{\Lambda}(p-1)-1$, we have $\rho = \frac{\lambda(p-1)-\Lambda(\alpha+1)}{b}>0$. Computing as above, for $x \not\in E$ we get

\begin{equation*}
\begin{split}
&\ \ {{{\cal P}^+_{\lambda,\Lambda|p}}}(D^2V_\alpha^\mu(x))+b|DV_\alpha^\mu(x)|\\
\leq&{\int_E {\cal P}^+_{\lambda,\Lambda|p}\left(D^2\Phi_\alpha(x-y)\right)+b\left|D\Phi_\alpha(x-y)\right|\,d\mu(y)}\\
=&\int_E\frac{(\alpha+\delta_{0,\alpha})\left({\Lambda (\alpha+1)-\lambda(p-1)}+b|x-y|\right)}{|x-y|^{\alpha+2}}\,d\mu(y)\\
=&\int_{E\cap B_\rho(x)}\frac{(\alpha+\delta_{0,\alpha})\left({\Lambda (\alpha+1)-\lambda(p-1)}+b|x-y|\right)}{|x-y|^{\alpha+2}}\,d\mu(y)\\
+&\,\int_{E\backslash B_\rho(x)}\frac{(\alpha+\delta_{0,\alpha})\left({\Lambda (\alpha+1)-\lambda(p-1)}+b|x-y|\right)}{|x-y|^{\alpha+2}}\,d\mu(y)\\
=:&\,I_1+I_2, \ \hbox{\rm say}\,.
\end{split}
\end{equation*}

\noindent For $y\in B_\rho(x)$ we have $\Lambda (\alpha+1)-\lambda(p-1)+b|x-y|\leq0$ and so $I_1 \le 0$. On the other hand, for $y\not\in B_\rho(x)$, using the assumption, we can estimate 
$$
\frac{(\alpha+\delta_{0,\alpha})\left(\Lambda (\alpha+1)-\lambda(p-1)+b|x-y|\right)}{|x-y|^{\alpha+2}} \le \frac{(\alpha+\delta_{0,\alpha})\,b}{\rho^{\alpha+1}} 
$$
and hence $I_2 \le K=K_{\lambda,\Lambda|p}(\alpha;b) \equiv \frac{(\alpha+\delta_{0,\alpha})\,b}{\rho^{\alpha+1}}$\,.\\
In the case $E$ compact we can therefore choose $v(x)=V^\mu_\alpha(x)$.\\ 

\noindent Now suppose $E=\bigcup_{m\in\mathbb N}E_m$ where $E_m$ are compact sets. The above argument provides a sequence of measure $\mu_m \in {\cal M}_1^+(E_m)$ such that (\ref{potenziale-supersoln}), (\ref{4eq2}) and (\ref{4eq3}) hold true with $E=E_m$ and $\mu=\mu_m$. \\
Setting $\omega_m(x) = V^{\mu_m}_\alpha(x)$ and
$c_m= \frac1{\max(\omega_m(x_0),1)}$,
the non-negative functions $\displaystyle v_N(x)=\sum_{m=1}^N\frac{c_m}{2^m}\,\omega_m(x)$ are viscosity solutions in $B_d$ of equation
\begin{equation*}
\begin{split}
&\,\,{\cal P}^+_{\lambda,\Lambda|p}(D^2v_N(x))+b|Dv_N(x)| \\
\leq&\,\sum_{m=1}^N\frac{c_m}{2^m}\left({\cal P}^+_{\lambda,\Lambda|p}(D^2\omega_m(x))+b|D\omega_m(x)|\right)\leq K.
\end{split}
\end{equation*}
In this way $\displaystyle v(x)=\lim_{N\to+\infty}v_N(x)$ is in turn a solution of (\ref{potenziale-supersoln}) as limit of a non-decreasing sequence of supersolutions,
and the proof is done observing that, by construction, 
$$v(x)=+\infty\;\;\text{on $E$}, \;\;v(x)\geq0\;\;\text{in $B_d$}$$
and 
$$v(x_0)=\sum_{m=1}^\infty\frac{c_m}{2^m}\,\omega_m(x_0)\leq1.$$ 
\end{proof}

\noindent\\ \emph{Proof of Theorem \ref{th1}}. For $\varepsilon>0$ consider the function 
$w_\varepsilon(x)=u(x)-\varepsilon v(x)$  where $v(x)$ is the 
function provided by Proposition \ref{potential}. By Lemma \ref{lemdifference} we infer that, for $c\geq0$ and $x\in\Omega$,
\begin{equation*}
\begin{split}
{\cal P}^+_{\lambda,\Lambda|p}(D^2w_\varepsilon)+b|Dw_\varepsilon|-cw_\varepsilon 
\geq&\,{{\cal P}^+_{\lambda,\Lambda|p}}(D^2u)+b|Du|-cu\\
-&\,\varepsilon\left({{\cal P}^+_{\lambda,\Lambda|p}}(D^2v)+b|Dv|\right)\\
\geq&\, f(x)-\varepsilon K.
\end{split}
\end{equation*}
From Propositions \ref{MPsmall} and \ref{MPc} we deduce the bounds
\begin{equation}\label{th1eq1}
w_\varepsilon(x)\leq\limsup_{y\to\partial\Omega}w_\varepsilon(y)+C\left\|(f-\varepsilon K)^-\right\|_\infty\quad\forall x\in\Omega
\end{equation}
in the case $c=0$ and
\begin{equation}\label{th1eq2}
w_\varepsilon(x)\leq\limsup_{y\to\partial\Omega}w_\varepsilon^+(y)+C\left\|(f-\varepsilon K)^-\right\|_\infty\quad\forall x\in\Omega
\end{equation}
in the case $c>0$. Since $\displaystyle\lim_{y\to E}w_\varepsilon(y)=-\infty$, the above inequalities yield respectively  
$$u(x_0)-\varepsilon v(x_0)\leq\limsup_{y\to E'}u(y)+C\left\|(f-\varepsilon K)^-\right\|_\infty$$
$$\hskip0.25cm u(x_0)-\varepsilon v(x_0)\leq\limsup_{y\to E'}u^+(y)+C\left\|(f-\varepsilon K)^-\right\|_\infty.$$
Letting $\varepsilon\to0^+$  we conclude the proof, since $x_0\in\Omega$ is arbitrary. \hfill$\Box$\\

\noindent From Theorem \ref{th1} we deduce the following extended comparison principle.\\

\begin{corollary}\label{corthm1}
Assume $\alpha^*=\frac\lambda\Lambda\,(p-1)-1\ge 0$, as in Theorem \ref{th1}. Let $\Omega\subseteq B_\delta$ be a bounded domain of $\mathbb R^n$. Suppose $\partial \Omega =E \cup E'$ where 
$E$ is an  $F_\sigma$-set such that ${\rm Cap}_\alpha(E)=0$ for $\alpha \in [0,\alpha^*]$. 
  Suppose that $u\in USC(\Omega)$ is bounded above, $v\in LSC(\Omega)$ is bounded from below and 
$$F(u,Du,D^2u)\geq f(x),\quad F(v,Dv,D^2v)\leq g(x)$$
in the viscosity sense in $\Omega$ and  $f,g \in C(\Omega)$. \\
Let us consider the following cases: $(i)$ $c=0$; $(ii)$ $c>0$\,.\\
Case $(i)$. If $b\delta <\lambda p$, then
\begin{equation}\label{CP}
\sup_\Omega (u-v)\leq\limsup_{y\to E'}(u-v)(y)+C\left\|(f-g)^-\right\|_\infty
\end{equation}
where $C$ is a constant depending only on $\lambda$, $p$, $b$ and $\delta$.\\
Case $(ii)$. For all $b \ge0$
\begin{equation}\label{CP+}
\sup_\Omega (u-v)\leq\limsup_{y\to E'}(u-v)^+(y)+C\left\|(f-g)^-\right\|_\infty
\end{equation}
where now the constant $C$ depends also on $c$.\\
Estimates (\ref{CP}) and (\ref{CP+}) hold true in general for $0\le \alpha < \alpha^*$ and also for $\alpha=\alpha^*$ if $b=0$.
\end{corollary}
 \begin{proof}
By virtue of Lemma \ref{lemdifference} the difference function $w=u-v$ is a viscosity subsolution in $\Omega$ of the equation
$${{\cal P}^+_{\lambda,\Lambda|p}}(D^2w)+b|Dw|-cw= f(x)-g(x).$$
The conclusion follows from Theorem \ref{th1}.
 \end{proof}

\section{Removable singularities}\label{Removable singularities}

\noindent This Section is devoted to the proof of the removability result Theorem \ref{threm} for the equation $F=f$
 in a domain $\Omega$ except for a closed subset, with respect to the relative topology, having suitable vanishing capacity.\\
Let $u(x)$ be a subsolution, bounded above on bounded sets, in $\Omega \backslash E$.
Following Harvey-Lawson \cite[Sections 3 and 6]{HL2}, we will show that the upper semicontinuous extension $U(x)$ of the subsolution $u(x)$ across $E$, as defined in Section \ref{Viscosity}, is in turn a subsolution of the equation $F=f$ in all $\Omega$.\\
A corresponding result holds true for supersolutions, bounded from below on bounded sets.\\

\noindent
\emph{Proof of Theorem \ref{threm}}. In what follows $U(x)$ and $V(x)$ will be the upper and the lower semicontinuous extension of $u(x)$ across $E$, which are respectively bounded above and below on bounded sets.\\
Let $E:=C\cap \Omega$ where $C$ is a closed set of $\mathbb R^n$, and $x_0 \in C\cap \Omega$ and $\delta$ be a positive number such that $b\delta < \lambda p$, as required by Theorem \ref{th1}, Case (i).\\
 Set also $\Omega_0:= B_\delta(x_0)\subset\Omega$ and $\overline\Omega_0$ the closure of $B_\delta(x_0)$ in the standard topology of $\mathbb R^n$.\\
For any $\varepsilon>0$ consider $w_\varepsilon=U-\varepsilon V^\mu_\alpha\in USC(\Omega_0)\cap C(\Omega_0\backslash E)$, where $V^\mu_\alpha$ is the potential provided by Proposition \ref{potential} which blows up on $C\cap\overline\Omega_0 \supset E\cap \Omega_0$.\\
 As in the proof of Theorem \ref{th1}, since $w_\varepsilon\equiv-\infty$ on $E\cap \Omega_0$, we infer that $w_\varepsilon(x)$ is a subsolution of equation
\begin{equation}\label{weps}
F\left(w_\varepsilon,Dw_\varepsilon,D^2w_\varepsilon\right)=f(x)-\varepsilon K\quad\text{in $\Omega_0$}.
\end{equation}
Moreover, the upper semicontinuous regularization $w^*<+\infty$ of the upper envelope $\displaystyle w=\sup_{0<\varepsilon < \varepsilon_0}w_{\varepsilon}$  is a viscosity subsolution of equation
\begin{equation}\label{wstar}
F\left(w^*,Dw^*,D^2w^*\right)=f(x)-\varepsilon_0 K\quad\text{in $\Omega_0$},
\end{equation}
by the stability results {of Section \ref{Viscosity}}. Since $U(x)=w^*(x)$, letting $\varepsilon_0 \to 0^+$, we conclude that $U(x)$ extends across $E$ the subsolution $u(x)$ in $\Omega_0$, namely
$$
F\left(U(x),DU(x),D^2U(x)\right)\geq f(x)\quad\text{in $\Omega_0$}
$$
and $U(x)=u(x)$ for $x \in \Omega_0\backslash E$.
On the other hand,
$w_{-\varepsilon}=U+\varepsilon V^\mu_\alpha\in LSC(\Omega_0)\cap C(\Omega_0\backslash E)$ 
is a supersolution of equation 
\begin{equation}\label{weps-super}
F\left(w_{-\varepsilon},Dw_{-\varepsilon},D^2w_{-\varepsilon}\right) = f(x)+\varepsilon K\quad\text{in $\Omega_0$},
\end{equation}
and the lower semicontinuous regularization $w_*>-\infty$ of the lower envelope $\displaystyle w=\inf_{0<\varepsilon < \varepsilon_0}w_{-\varepsilon}$ is a viscosity supersolution of equation
\begin{equation}\label{wstar-super}
F\left(w_*,Dw_*,D^2w_*\right)= f(x) +\varepsilon_0 K\quad\text{in $\Omega_0$}.
\end{equation}
As before, since $V(x)=w_*(x)$, letting $\varepsilon_0 \to 0^+$, we conclude that $V(x)$  extends across $E$ the supersolution $u(x)$, namely
$$
F\left(V(x),DV(x),D^2V(x)\right)\leq f(x) \quad\text{in }\Omega_0\,.
$$
and $V(x)=u(x)$ for $x \in \Omega_0\backslash E$.\\
Hence $U(x)$ and  $V(x)$  are respectively subsolution  and supersolution of equation $F=f$ such that $U(x)=u(x)=V(x)$ in $\Omega_0\backslash E$. \\
We claim that $U(x_0)=V(x_0)$. Indeed, by semicontinuity, we have $V(x_0)\le U(x_0)$. On the other hand, since $U(x)=V(x)$ on $\partial \Omega_0$ except a closed set of vanishing capacity, by the extended comparison principle of Corollary \ref{corthm1}, we get $U(x)\le V(x)$ in $\Omega_0\equiv B_\delta(x_0)$  and in particular $U(x_0)\le V(x_0)$. Therefore we conclude that $U(x_0)=V(x_0)$, as claimed.\\
Since $x_0\in E$ is arbitrary, then $U(x)=V(x)$ on $E$, so that, by continuity of $u(x)$, we can conclude that $U(x)=V(x)$ for all $x \in \Omega$. Then the function $W(x)=U(x)=V(x)$ is a viscosity solution of equation $F=f$ in $\Omega$ which extends the given solution $u(x)$ in $\Omega\backslash E$, as asserted.\hfill$\Box$ 

\noindent {\bf Acknowledgements}

\noindent The authors wish to thank the Gruppo  Nazionale per l'Analisi Matematica, la Probabilit\`a e le loro  Applicazioni (GNAMPA) of the Istituto Nazionale di Alta Matematica  (INdAM) that partially supported this work.

\end{document}